\documentclass[12pt]{amsart}

\usepackage[
  letterpaper,
  textwidth=6.2in,
  textheight=8.6in,
  hcentering,
  vcentering,
  headheight=14pt,
  headsep=0.2in
]{geometry}

\usepackage{amssymb}
\usepackage{mathtools}
\usepackage{mathrsfs}

\usepackage{graphicx}
\usepackage{epstopdf}
\usepackage{xcolor}

\usepackage{scalerel}
\usepackage{stackengine}

\usepackage{tikz}
\usepackage{tikz-cd}
\usepackage[all]{xy}

\xyoption{matrix}
\xyoption{arrow}

\usetikzlibrary{
  arrows,
  arrows.meta,
  decorations.pathmorphing,
  backgrounds,
  positioning,
  fit,
  petri
}

\tikzset{
  help lines/.style={
    step=#1cm,
    very thin,
    color=gray
  },
  help lines/.default=.5,
  thick grid/.style={
    step=#1cm,
    thick,
    color=gray
  },
  thick grid/.default=1
}

\usepackage{hyperref}

\hypersetup{
  colorlinks=true,
  citecolor=blue,
  linkcolor=black,
  urlcolor=blue
}

\newtheorem{theorem}{Theorem}[section]
\newtheorem{lemma}[theorem]{Lemma}
\newtheorem{corollary}[theorem]{Corollary}
\newtheorem{proposition}[theorem]{Proposition}

\theoremstyle{definition}
\newtheorem{definition}[theorem]{Definition}
\newtheorem{example}[theorem]{Example}

\theoremstyle{remark}
\newtheorem{remark}[theorem]{Remark}
\numberwithin{equation}{section}
\newtheoremstyle{case}{}{}{}{}{}{:}{ }{}
\theoremstyle{case}

\DeclareMathOperator{\Hom}{Hom}

\newcommand{\field}[1]{\mathbb{#1}}
\newcommand{\ZZ}{\ensuremath{{\field{Z}}}}

\newcommand{\NN}{\ensuremath{{\field{N}}}}

\newcommand{\commentout}[1]{}

\newcommand{\cC}{\ensuremath{{\mathcal{C}}}}

\newcommand{\cN}{\ensuremath{{\mathcal{N}}}}
\newcommand{\cP}{\ensuremath{{\mathcal{P}}}}

\begin{document}
\title{Finite Quotients of Right-angled Artin Groups}
\author{Zihao Liu}
\address{\itshape Zihao Liu, Department of Mathematics, Rice University, Houston, TX, 77005}
\email{zihaoliu@rice.edu}

\subjclass[2020]{Primary 20F36; Secondary 20E18, 20E26.}

\keywords{Artin groups, right-angled Artin groups, profinite completions, pro-$p$ completions, profinite rigidity}

\begin{abstract}
We prove that right-angled Artin groups (RAAGs) are profinitely rigid relative to the class of Artin groups. More precisely, if $A_\Gamma$ is an arbitrary Artin group and $A_\Lambda$ is a right-angled Artin group such that their profinite completions are isomorphic, i.e.,
$\widehat{A_\Gamma}\cong\widehat{A_\Lambda},$
then $\Gamma$ is right-angled and $\Gamma\cong\Lambda$ as labeled graphs. The proof combines the structure of pro-$p$ completions of Artin groups with finite Coxeter quotients to recover the defining graph.
\end{abstract}

\maketitle
\tableofcontents
\section{Introduction}

To what extent is a finitely generated group determined by its finite quotients? This question has attracted considerable attention in recent years, drawing on ideas from geometric group theory, low-dimensional topology, arithmetic, and geometry. For a finitely generated group $G$, denote by $\cC(G)$ the set of isomorphism classes of finite quotients of $G$. Two groups $G$ and $H$ are said to have the same finite quotients if $\cC(G)=\cC(H)$.
It is by now a standard fact that, for finitely generated groups, having the same finite quotients is equivalent to an isomorphism $\widehat G\cong\widehat H$ of their profinite completions (see Dixon--Formanek--Poland--Ribes~\cite{DFPR82}). We shall therefore pass freely between these two formulations.

Let $\cP$ be a class of finitely generated groups. A group $G\in\cP$ is said to be \emph{profinitely rigid relative to $\cP$} if, for every $H\in\cP$, one has $G\cong H$ whenever $\cC(G)=\cC(H)$. A finitely generated residually finite group is called \emph{(absolutely) profinitely rigid} if it is profinitely rigid relative to the class of all finitely generated residually finite groups.

In this paper we study profinite rigidity for right-angled Artin groups (RAAGs) within the larger class of Artin groups. RAAGs form one of the fundamental classes of groups in geometric group theory. They interpolate between finitely generated free groups and finitely generated free abelian groups, and are precisely the Artin groups whose defining graphs have all edge labels equal to $2$.

A theorem of Kropholler--Wilkes~\cite{KW16} shows that RAAGs are already rigid from a much smaller collection of finite quotients: for every prime number $p$, two right-angled Artin groups with isomorphic pro-$p$ completions, equivalently, with the same finite $p$-group quotients up
to isomorphism, have isomorphic defining graphs. Thus, once two groups are known to be RAAGs, a single pro-$p$ completion determines their defining graphs. The problem considered here is therefore not rigidity within the class of RAAGs, but whether right-angledness itself can be recognized among all Artin groups.

(Absolute) profinite rigidity fails in general for RAAGs. Indeed, the classical construction of Platonov--Tavgen'~\cite{PT86} gives a finitely generated but not finitely presented subgroup $P<F_4\times F_4$ for which the inclusion induces an isomorphism $\widehat P\cong\widehat{F_4\times F_4}.$ In particular, $P\not\cong F_4\times F_4$. Since $F_4\times F_4$ is the RAAG associated to the complete bipartite graph $K_{4,4}$, this shows that RAAGs need not be (absolutely) profinitely rigid. This makes relative profinite rigidity within the class of Artin groups the natural formulation of the problem.

There is a closely analogous phenomenon for Coxeter groups.
Kropholler--Wilkes proved that right-angled Coxeter groups are distinguished from one another by their pro-$2$
completions~\cite{KW16}, while more recently
Corson--Hughes--M\"oller--Varghese proved that every right-angled Coxeter group is profinitely rigid relative to the class of all Coxeter groups~\cite{CHMV26} . These results suggest the corresponding
relative rigidity problem for RAAGs.

Although residual finiteness is not known for arbitrary Artin groups, this does not prevent one from studying their finite quotients. That there are nontrivial finite quotients can be seen by sending every standard generator to $1$, which defines an epimorphism from any nontrivial Artin group onto $\ZZ$, and hence onto $\ZZ/n\ZZ$ for every $n\geq2$. We therefore formulate the problem directly in terms of finite quotients, or equivalently, since Artin groups are finitely generated, in terms of their profinite completions.

Our main problem is the following: let $A_\Lambda$ be a RAAG and let $A_\Gamma$ be an arbitrary Artin group. If
\[
\widehat{A_\Gamma}\cong\widehat{A_\Lambda},
\]
must $\Gamma$ itself be right-angled? We answer this question affirmatively and, in fact, recover the entire defining graph.

\begin{theorem}\label{thm:main}
Let $A_\Gamma$ be an Artin group and let $A_\Lambda$ be a
right-angled Artin group. If
\[
\widehat{A_\Gamma}\cong\widehat{A_\Lambda},
\]
then $\Gamma$ is right-angled and $\Gamma\cong\Lambda$ as labeled graphs. In particular, $A_\Gamma\cong A_\Lambda$. Thus right-angled Artin groups are profinitely rigid relative to the class of Artin groups.
\end{theorem}

Note also that, within the class of Artin groups, Nyberg--Brodda
recently proved that the braid groups $B_n$ are not absolutely
profinitely rigid for $n\geq4$~\cite{NB26}.

\medskip

We now describe the main ingredients in the proof. For a group $G$ and a prime number $p$, we denote by $\widehat{G}^{(p)}$ its pro-$p$ completion, namely the inverse limit of its finite $p$-group quotients. A recent theorem of Escart\'in Ferrer--Leoni--Mart\'inez P\'erez associates to every Artin graph $\Gamma$ and every prime number $p$ an even labeled graph $\Gamma_p$, which we call the $p$-part of $\Gamma$, and proves that
\[
\widehat{A_\Gamma}^{(p)}
\cong
\widehat{A_{\Gamma_p}}^{(p)}.
\]
See Section~\ref{sec:preliminaries} for the precise definition of $\Gamma_p$. Roughly speaking, $\Gamma_p$ is obtained by collapsing
the components generated by odd-labeled edges, replacing an even label $2kp^t$, where $\gcd(k,p)=1$, by $2p^t$, and retaining the smallest label when parallel edges arise~\cite{ELMP26}.

Our first step is to show that right-angledness of the $p$-part is itself detected by the pro-$p$ completion. Suppose that $\Gamma_p$ contains an edge labeled $2p^t$ with $t\geq1$. The corresponding rank-two standard parabolic subgroup is the dihedral Artin group
\[
D_{p^t}
=
\langle a,b \mid (ab)^{p^t}=(ba)^{p^t}\rangle.
\]
We show that $\widehat{D_{p^t}}^{(p)}$ is nonabelian, has nontrivial center, and is topologically generated by the images of $a$ and $b$; that is, these two elements generate a dense subgroup. Since standard parabolic subgroups of even Artin groups are retracts,
its pro-$p$ completion embeds as a closed subgroup of the relevant ambient pro-$p$ completion. On the other hand, Casals-Ruiz--Pintonello--Zalesskii proved that every closed subgroup of a pro-$p$ RAAG that is topologically generated by two elements is either free pro-$p$ or free abelian
pro-$p$~\cite{CPZ25}. The subgroup above can be neither, resulting in a contradiction. Once $\Gamma_p$ is known to be right-angled, the theorem of Kropholler--Wilkes~\cite{KW16} recovers its defining graph.

In particular, if $A_\Lambda$ is a RAAG, then for every prime
$p$, $\widehat{A_\Gamma}^{(p)}\cong\widehat{A_\Lambda}^{(p)}$ if and only if $\Gamma_p\cong\Lambda$
as labeled graphs. Thus, the pro-$p$ completion
detects exactly whether the $p$-part of an Artin graph is the given right-angled graph.

We briefly indicate how this leads to Theorem~\ref{thm:main}. Suppose that
\[
\widehat{A_\Gamma}\cong\widehat{A_\Lambda},
\]
where $A_\Lambda$ is right-angled. Passing to maximal pro-$p$ quotients and applying the preceding result gives
\[
\Gamma_p\cong\Lambda
\]
for every prime $p$. However, the construction of $\Gamma_p$ collapses the components generated by odd-labeled edges, so this pro-$p$ information alone does not rule out the presence of odd labels in $\Gamma$.

The remaining step uses genuinely profinite information. There is a canonical quotient associated to the odd-collapse of $\Gamma$; see
Section~\ref{sec:detecting-odd-edges} for its precise definition and basic properties.  Under the hypothesis of Theorem~\ref{thm:main}, this odd-collapse
quotient has the same profinite completion as $A_\Gamma$. A general factorization lemma then implies that every homomorphism from $A_\Gamma$ to a finite group factors through this quotient. On the other hand, $A_\Gamma$ admits a canonical epimorphism onto its
associated Coxeter group $W_\Gamma$, and residual finiteness of Coxeter
groups then provides finite quotients separating the endpoints of every
odd-labeled edge. This contradiction shows that $\Gamma$ has no odd-labeled edges. Once $\Gamma$ is even, the fixed-prime recognition
result above forces every edge label to be $2$, and therefore $\Gamma\cong\Lambda.$

Along the way, we also determine the information detected
simultaneously by all pro-$p$ completions relative to a RAAG. The result is expressed in terms of the odd-collapse graph and the arithmetic data associated with the even-labeled edges between odd components; see Section~\ref{sec:detecting-odd-edges} for the precise statement.

The fixed-prime result suggests a broader reconstruction problem: given arbitrary Artin groups $A_\Gamma$ and $A_\Delta$, does an
isomorphism
\[
\widehat{A_\Gamma}^{(p)}
\cong
\widehat{A_\Delta}^{(p)}
\]
necessarily imply
\[
\Gamma_p\cong\Delta_p?
\]
The results of the present paper answer this question when one of the
two $p$-parts is right-angled.

The paper is organized as follows. In Section~\ref{sec:preliminaries} we recall background on Artin and Coxeter groups, profinite and pro-$p$ completions, and
the rigidity results used later. In Section~\ref{sec:recognition-p-parts} we establish the rank-two pro-$p$ obstruction and the fixed-prime recognition theorem.
In Section~\ref{sec:detecting-odd-edges} we study odd components and the information detected by the collection of all pro-$p$ completions, and prove the finite-quotient obstruction needed to detect odd labels. In
Section~\ref{sec:relative-profinite-rigidity} we complete the proof of Theorem~\ref{thm:main}.

\medskip

\noindent\textbf{Role of AI.}
During the preparation of this manuscript, the author used OpenAI's ChatGPT (GPT-5.6) as an assistive tool for language editing, organization and presentation, literature searches, translation, and related aspects of manuscript preparation. The author reviewed all resulting material and takes full responsibility for the mathematical content and the final manuscript.

\medskip

\noindent\textbf{Acknowledgments.}
The author would like to thank his Ph.D. advisor, Alan Reid, for many helpful discussions, encouragement, and careful feedback on the proofs and on numerous versions of this manuscript throughout the entire process of writing this paper. The author is also grateful to Hongbin Sun for helpful comments on earlier drafts.

\section{Preliminaries}\label{sec:preliminaries}
\subsection{Artin groups and Coxeter groups}
Let $\Gamma=(V(\Gamma),E(\Gamma))$ be a finite simplicial graph with each of its edges labelled by an integer $\ge 2$. Denote by $\{s_i,s_j\}\in E$ an edge spanned by $s_i$ and $s_j$. The
\emph{Artin group with defining graph $\Gamma$}, denoted by $A_{\Gamma}$, is given by the following presentation:
\[
\langle s_i\in V
\ |\
\underbrace{s_is_js_i\cdots}_{m_{ij}}=\underbrace{s_js_is_j\cdots}_{m_{ij}}
\text{for each $\{s_i,s_j\}\in E(\Gamma)$ labelled by $m_{ij}$}
\rangle.
\]

We say that $A_\Gamma$ is \emph{right-angled} if $m_{ij}=2$ for every $\{s_i,s_j\}\in E(\Gamma)$, and \emph{even} if every edge label is even.
\begin{example}
The following are standard examples of Artin groups.
\begin{enumerate}
\item If $\Gamma$ has no edges, then $A_{\Gamma}$ is the free group on the vertex set $V(\Gamma)$.
\item If $\Gamma$ is complete and every edge has label $2$, then $A_\Gamma$ is the free abelian group of rank $|V(\Gamma)|$.
\item Let $n\geq 2$, and let $\Gamma$ be the complete graph on the
vertex set $\{s_1,\ldots,s_{n-1}\}$, with edge labels
$$
m_{ij}
=
\begin{cases}
3, & |i-j|=1,\\
2, & |i-j|\geq 2.
\end{cases}
$$
Then
\[
A_\Gamma
=
\left\langle
s_1,\ldots,s_{n-1}
\ \middle|\
\begin{array}{ll}
s_is_{i+1}s_i=s_{i+1}s_is_{i+1},
    & 1\leq i\leq n-2,\\
s_is_j=s_js_i,
    & |i-j|\geq 2
\end{array}
\right\rangle,
\]
which is the braid group $B_n$ on $n$ strands.

\item Let $\Gamma$ consist of two vertices $a$ and $b$ joined by a
single edge labeled $m\geq 2$. The associated Artin group is the
\emph{dihedral Artin group} of type $I_2(m)$, given by
$$
A_{I_2(m)}
=
\langle
a,b
\ |\
\underbrace{aba\cdots}_m
=
\underbrace{bab\cdots}_m
\rangle.
$$
Equivalently, if $m=2k$ is even, then
$$
A_{I_2(2k)}
=
\langle a,b \ | \ (ab)^k=(ba)^k\rangle,
$$
whereas if $m=2k+1$ is odd, then
$$
A_{I_2(2k+1)}
=
\langle a,b\ | \ (ab)^ka=(ba)^kb\rangle.
$$
\end{enumerate}
\end{example}

Given an Artin group $A_\Gamma$, its \emph{associated Coxeter group} is the quotient by the subgroup normally generated by the squares of generators of $A_{\Gamma}$.
Equivalently, $W_\Gamma$ has the presentation
$$
W_\Gamma
=
\langle
s_i\in V(\Gamma)
\ |\
s_i^2=1,\ 
\underbrace{s_is_js_i\cdots}_{m_{ij}}
=
\underbrace{s_js_is_j\cdots}_{m_{ij}}
\text{ for each $\{s_i,s_j\}\in E(\Gamma)$ labelled by $m_{ij}$}
\rangle.
$$
Since the generators are involutions, the Artin relation associated
with an edge labeled $m_{ij}$ is equivalent to
$(s_is_j)^{m_{ij}}=1$. We say that $W_\Gamma$ is
\emph{right-angled} if every edge of $\Gamma$ has label $2$.

\begin{example}
The associated Coxeter group of the braid group $B_n$ is the symmetric
group $S_n$, under the quotient map
$\sigma_i\mapsto(i\ \ i+1)$. If $\Gamma$ consists of two vertices
joined by an edge labeled $m\geq 2$, then
\[
W_\Gamma
=
\langle a,b\mid a^2=b^2=1,\ (ab)^m=1\rangle
\]
is the dihedral group of order $2m$.
\end{example}

Let $A_\Gamma$ be an Artin group with vertex set $V$. Given a subset $X\subseteq V$ , we define the \emph{standard parabolic subgroup} $A_X$ as the subgroup $\langle X\rangle\subseteq A_{\Gamma}$. By \cite{VdL83}, \cite{Paris97}, this is well defined since $A_X$ is canonically isomorphic to the Artin group defined via the full subgraph spanned by $X$. Any conjugate of a standard parabolic subgroup is called \emph{parabolic}.

\begin{lemma}\label{lem:standard-retract-even}
Let $A_\Gamma$ be an even Artin group and $A_X$ be a standard parabolic subgroup where $X\subseteq V(\Gamma)$. Then the map $p_X:A_{\Gamma}\to A_X$ induced by $p_X(v)=v$ for all $v\in X$ and $p_X(w)=1$ for $w\in V(\Gamma)\setminus X$ is a well-defined homomorphism and retraction, i.e., $p_X|_{A_X}=\operatorname{id}_{A_X}$.
\end{lemma}

\begin{proof}
We first check that $p_X$ preserves
the Artin relations. Let $\{s_i,s_j\}\in E(\Gamma)$ be labeled
by $m_{ij}$. Since $A_\Gamma$ is even, $m_{ij}=2k$ for some $k\ge 1$, and the corresponding relation becomes
\[
(s_is_j)^k=(s_js_i)^k.
\]

If both $s_i, s_j\in X$, then this is one of the
defining relations of $A_X$. If $s_i,s_j\in V(\Gamma)\setminus X$, we have that $p_X(s_i)=p_X(s_j)=1$. Finally, suppose that there is exactly one of the two generators in $X$, say $s_i\in X$ and
$s_j\notin X$. Then we have
\[
p_X((s_is_j)^k)=s_i^k=p_X((s_js_i)^k).
\]
Thus every defining relation of $A_\Gamma$ is preserved, so $p_X$ is
a well-defined homomorphism.

Since $p_X$ fixes every generator in $X$, we have
$p_X\circ\iota_X=\operatorname{id}_{A_X}$ where $\iota_X:A_X\hookrightarrow A_{\Gamma}$ is the inclusion. It follows that
$\iota_X$ is injective and its image is a retract of $A_\Gamma$.
\end{proof}

\medskip

A basic question is whether $A_\Gamma\cong A_\Delta$ implies that the defining labeled graphs $\Gamma$ and $\Delta$ are isomorphic. This is known for RAAGs by a theorem of Droms~\cite{Droms87}, and is conjectured for arbitrary even Artin groups. Blasco-García and París established the conjecture ~\cite{BGP22} when every finite edge label belongs to
$\{2c\}\cup\{2d^r \ | \ r\geq 1\}$ for some integers $c\geq 1$ and $d\geq 2$ with $\gcd(c,d)=1$.

To put this paper in contrast, we are considering the corresponding profinite recognition problem: whether an Artin group whose profinite completion is isomorphic to that of a RAAG must itself be right-angled.

\subsection{Profinite and pro-\texorpdfstring{$p$}{p} completions}
A group $G$ is \emph{residually finite} if, for every $g\in G\setminus\{1\}$, there exist a finite group $Q$ and a homomorphism $\varphi\colon G\to Q$ such that $\varphi(g)\neq 1$. By Mal'cev's theorem~\cite{Mal65}, every finitely generated linear group is residually finite, so every Coxeter group is residually finite since Coxeter groups are linear; see \cite[Corollary~6.12.11]{Dav08}.

The residual finiteness of arbitrary Artin groups remains open.
Nevertheless, it is known for several important classes. Artin groups of spherical type are linear, and hence residually finite, by work of Cohen--Wales and Digne~\cite{CW02,Digne03}. RAAGs are also linear and therefore residually finite; see Humphries~\cite{Hum94} or Hsu--Wise~\cite{HW99}. Moreover, even Artin groups of FC type are poly-free and residually finite \cite{BGMPP19}. Further classes of two-dimensional Artin groups were shown to be residually finite by Jankiewicz~\cite{Jan22}.

\medskip

Now let $\cN$ be the set of
finite-index normal subgroups of $G$, ordered by reverse inclusion. Whenever $N\subseteq M\in\cN$, there is a natural quotient homomorphism $G/N\to G/M$, and these maps form an inverse system.

\begin{definition}
The \emph{profinite completion} of $G$ is given by the inverse limit
\[
\widehat{G}\coloneq\varprojlim_{N\in\cN} G/N.
\]
Equivalently, after equipping each finite quotient $G/N$ with the
discrete topology, $\widehat{G}$ is the closed subgroup of
$\prod_{N\in\cN}G/N$ consisting of all compatible
families. In particular, $\widehat{G}$ is a compact Hausdorff and
totally disconnected topological group.

For a prime number $p$, let $\cN_p$ be the set of normal
subgroups $N\trianglelefteq G$ such that $G/N$ is a finite $p$-group.
The \emph{pro-$p$ completion} of $G$ is given by
\[
\widehat{G}^{(p)}\coloneq \varprojlim_{N\in\cN_p} G/N.
\]
Equivalently, the inverse limit is taken over all normal subgroups
$N\trianglelefteq G$ for which $[G:N]$ is a power of $p$.
\end{definition}

There is a canonical homomorphism $i:G\to\widehat{G}$ given by $i(g)=(gN)_{N\in\cN}$. Note that $G$ is residually finite if and only if $i$ is injective. 

The canonical homomorphism $i$ satisfies the following universal property: If $H$ is a profinite group and $\varphi: G\to H$ is a continuous homomorphism, then there exists a unique continuous homomorphism $\widehat{\varphi}:\widehat{G}\to H$ such that
$\widehat{\varphi}\circ i=\varphi$. 

\[
\begin{tikzcd}[column sep=3.5em, row sep=2.8em]
G \arrow[r,"i"] \arrow[dr,"\varphi"']
&
\widehat{G} \arrow[d,"\widehat{\varphi}"]
\\
&
H.
\end{tikzcd}
\]
The analogous universal property holds for $\widehat{G}^{(p)}$ when $H$ is a pro-$p$ group.

\begin{proposition}\label{prop:hom-finite-quotients}
Let $G$ be a group and let $Q$ be a finite group equipped with the discrete topology. Precomposition with the canonical homomorphism
$i: G\to\widehat{G}$ induces a natural bijection
\[
\begin{aligned}
\Hom_{\mathrm{cont}}(\widehat{G},Q)
&\to
\Hom(G,Q),\\
\psi
&\to
\psi\circ i.
\end{aligned}
\]
where $\Hom_{\mathrm{cont}}(\widehat{G},Q)$ is the set of continuous homomorphisms from $\widehat{G}$ to $Q$. Consequently, if $\widehat{G}\cong\widehat{H}$, then
$|\Hom(G,Q)|=|\Hom(H,Q)|$ for every finite group $Q$.
\end{proposition}

\begin{proof}
Let $\varphi:G\to Q$ be a homomorphism. Since $Q$ is finite, it is profinite. Hence, by the universal property of the profinite completion, there exists a unique continuous homomorphism
$\widehat{\varphi}: \widehat{G}\to Q$ such that $\widehat{\varphi}\circ i=\varphi$.

Conversely, every continuous homomorphism
$\psi\colon\widehat{G}\to Q$ gives a
homomorphism $\psi\circ i: G\to Q$. These two constructions are clearly mutually inverse, so we have a bijection $\Hom_{\mathrm{cont}}(\widehat{G},Q)\to
\Hom(G,Q)$.

Now if $\widehat{G}\cong\widehat{H}$, then 
\[
|\Hom(G,Q)|=|\Hom_{\mathrm{cont}}(\widehat{G},Q)|=|\Hom_{\mathrm{cont}}(\widehat{H},Q)|=|\Hom(H,Q)|.
\]

\end{proof}

For a group $G$, let $\cC(G)$ denote the set of isomorphism classes of finite quotients of $G$. The following theorem shows that, for a finitely generated group $G$, $\cC(G)$ determines its profinite completion.

\begin{theorem}[Dixon--Formanek--Poland--Ribes~\cite{DFPR82}]
\label{thm:finite-quotients-completion}
For finitely generated groups $G$ and $H$, $\cC(G)=\cC(H)$ if and only if $\widehat G\cong\widehat H$.
\end{theorem}

\begin{remark}
A profinite group $G$ is \emph{topologically generated} by a subset $S\subseteq G$ if $\overline{\langle S\rangle}=P$, and is
\emph{topologically finitely generated} if such an $S$ can be chosen to be finite. If $G$ and $H$ are finitely generated, then $\widehat G$ and $\widehat H$ are topologically finitely generated. By the strong completeness theorem of Nikolov--Segal, any abstract isomorphism $\widehat G\to\widehat H$ is automatically a topological isomorphism~\cite{NS07a,NS07b}.
\end{remark}

\subsection{Rigidity and subgroup results for pro-\texorpdfstring{$p$}{p} RAAGs}

A profinite group is a \emph{pro-$p$ group} if it is isomorphic, as a topological group, to an inverse limit of finite $p$-groups. Equivalently, all of its finite continuous quotients are $p$-groups.

For a profinite group $G$, define
\[
R_p(G)
\coloneq
\bigcap_{\substack{N\trianglelefteq_o G\\
G/N\text{ is a finite }p\text{-group}}}N,
\]
where $N\trianglelefteq_o G$ means that $N$ is an open normal subgroup of 
$G$. The quotient
\[
G_{(p)}\coloneq G/R_p(G)
\]
is called the \emph{maximal pro-$p$ quotient} of $G$ and it is easy to check that 
\[
G_{(p)}= \varprojlim_{\substack{N\trianglelefteq_o G\\G/N\text{ is a finite }p\text{-group}}}
 G/N.
\]

\begin{lemma}\label{lem:maximal-pro-p-quotient}
Let $G$ be a group and let $p$ be a prime number. Then
\[
(\widehat G)_{(p)}\cong\widehat G^{(p)}.
\]
Consequently, an isomorphism $\widehat G\cong\widehat H$ induces an
isomorphism
\[
\widehat G^{(p)}
\cong
\widehat H^{(p)}
\]
for every prime $p$.
\end{lemma}

\begin{proof}
By the universal property of the profinite completion, the finite continuous $p$-group quotients of $\widehat G$ are naturally the same as the finite $p$-group quotients of $G$. So $(\widehat{G})_{(p)}$ and $\widehat{G}^{(p)}$ are inverse limits over the same system of finite $p$-group quotients, and therefore, $(\widehat G)_{(p)}\cong\widehat G^{(p)}$. 

Now if $\widehat{G}\cong\widehat{H}$, then passing to maximal pro-$p$ quotients gives 
\[
(\widehat{G})_{(p)}\cong (\widehat{H})_{(p)}.
\]
and thus, $\widehat{G}^{(p)}\cong\widehat{H}^{(p)}$.
\end{proof}

Kropholler and Wilkes strengthened Droms' rigidity theorem by showing
that, for any fixed prime $p$, the defining graph of a right-angled
Artin group is already determined by its pro-$p$ completion.

\begin{theorem}[Kropholler--Wilkes {\cite[Theorem~4]{KW16}}]
\label{thm:KW-pro-p-rigidity}
Let $A_\Gamma$ and $A_\Lambda$ be right-angled Artin groups with
finite defining graphs $\Gamma$ and $\Lambda$, and let $p$ be a
prime number. Then
\[
\widehat{A_\Gamma}^{(p)}
\cong
\widehat{A_\Lambda}^{(p)}
\]
if and only if $\Gamma\cong\Lambda$.
\end{theorem}

As an immediate consequence, RAAGs are profinitely rigid relative to the class of right-angled Artin groups:

\begin{corollary}\label{cor:RAAG-profinite-rigidity}
Let $A_\Gamma$ and $A_\Lambda$ be right-angled Artin groups with
finite defining graphs $\Gamma$ and $\Lambda$. Then
$\widehat{A_\Gamma}\cong\widehat{A_\Lambda}$
if and only if $\Gamma\cong\Lambda$. Equivalently,
$\widehat{A_\Gamma}\cong\widehat{A_\Lambda}$ if and only if $A_\Gamma\cong A_\Lambda$.
\end{corollary}

\begin{proof}
If $\widehat{A_\Gamma}\cong\widehat{A_\Lambda}$, by Lemma~\ref{lem:maximal-pro-p-quotient}, we have that
\[
\widehat{A_\Gamma}^{(p)}
\cong
\widehat{A_\Lambda}^{(p)}
\]
for every prime number $p$. Then Theorem~\ref{thm:KW-pro-p-rigidity} implies that $\Gamma\cong\Lambda$. The converse follows from the functoriality of profinite completion.
\end{proof}

\bigskip

Given an Artin group $A_\Gamma$ and a prime number $p$, let $\Gamma_p$ be
the labeled graph obtained from $\Gamma$ by the following operations:
\begin{enumerate}
\item identify the endpoints of every odd-labeled edge and remove all
loops created by these identifications;
\item replace every label of the form $2kp^t$, where $t\geq 0$ and
$\gcd(k,p)=1$, by $2p^t$;
\item whenever multiple edges arise, replace them by a single edge
carrying the smallest of their labels.
\end{enumerate}
Thus $\Gamma_p$ is an even labeled graph whose edge labels have the form $2p^t$, and is called the \emph{$p$-part of $A_{\Gamma}$}.

\begin{example}\label{ex:p-part}
Let $\Gamma$ be the labeled triangle on $\{a,b,c\}$ whose edges
$\{a,b\}$, $\{a,c\}$, and $\{b,c\}$ have labels $3$, $12$, and $4$,
respectively. Since $\{a,b\}$ has odd label, its endpoints are
identified in every $\Gamma_p$; write $x=[a]=[b]$. The remaining
edges become parallel edges joining $x$ to $c$. For $p=2$, both labels
reduce to $4$, so $\Gamma_2$ has one edge labeled $4$. For $p=3$, the
labels reduce to $6$ and $2$, and retaining the smaller one gives an
edge labeled $2$. For every prime $p\geq 5$, both labels reduce to
$2$. Thus $\Gamma_3$ and $\Gamma_p$ are isomorphic as labeled graphs
for every $p\geq 5$, as illustrated in
Figure~\ref{fig:gamma-p-example}.
\end{example}

\begin{figure}[ht]
\centering
\resizebox{0.98\textwidth}{!}{%
\begin{tikzpicture}[
    vertex/.style={circle, fill=black, inner sep=1.8pt},
    every node/.style={font=\small}
]
\begin{scope}[shift={(0,0)}]
    \node[vertex, label=above:$a$] (a) at (0,1.5) {};
    \node[vertex, label=below:$b$] (b) at (0,0) {};
    \node[vertex, label=right:$c$] (c) at (1.65,0.75) {};
    \draw (a)--(b);
    \draw (a)--(c);
    \draw (b)--(c);
    \node[left=2pt] at (0,0.75) {$3$};
    \node[above right=-1pt] at (0.70,1.10) {$12$};
    \node[below right=-1pt] at (0.70,0.40) {$4$};
    \node at (0.82,-0.55) {$\Gamma$};
\end{scope}
\begin{scope}[shift={(3.1,0.75)}]
    \node[vertex] (x2) at (0,0) {};
    \node[vertex, label=right:$c$] (c2) at (1.55,0) {};
    \draw (x2)--(c2);
    \node[above=3pt] at (x2) {$x=[a]=[b]$};
    \node[below=2pt] at (0.78,0) {$4$};
    \node at (0.78,-0.72) {$\Gamma_2$};
\end{scope}
\begin{scope}[shift={(6.3,0.75)}]
    \node[vertex] (x3) at (0,0) {};
    \node[vertex, label=right:$c$] (c3) at (1.55,0) {};
    \draw (x3)--(c3);
    \node[above=3pt] at (x3) {$x=[a]=[b]$};
    \node[below=2pt] at (0.78,0) {$2$};
    \node at (0.78,-0.72) {$\Gamma_3$};
\end{scope}
\begin{scope}[shift={(9.5,0.75)}]
    \node[vertex] (xp) at (0,0) {};
    \node[vertex, label=right:$c$] (cp) at (1.55,0) {};
    \draw (xp)--(cp);
    \node[above=3pt] at (xp) {$x=[a]=[b]$};
    \node[below=2pt] at (0.78,0) {$2$};
    \node at (0.78,-0.72) {$\Gamma_p,\ p\geq 5$};
\end{scope}
\end{tikzpicture}%
}
\caption{The graphs $\Gamma_p$ associated with the labeled graph $\Gamma$.}
\label{fig:gamma-p-example}
\end{figure}
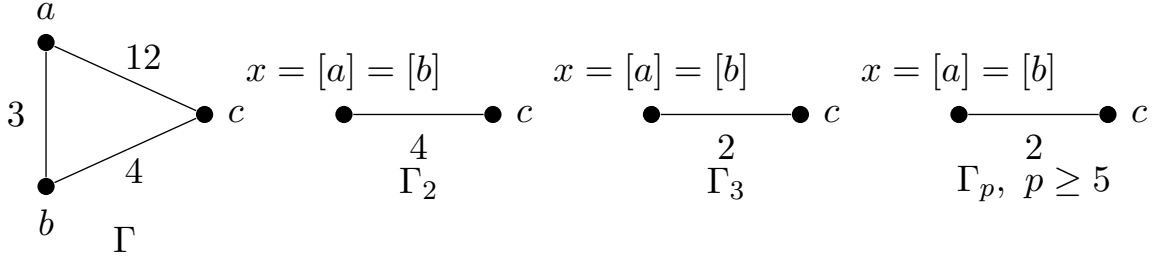

Escart\'in Ferrer, Leoni, and Mart\'inez P\'erez recently established
the following description of the pro-$p$ completion of an arbitrary
Artin group.

\begin{theorem}[Escart\'in Ferrer--Leoni--Mart\'inez P\'erez
{\cite[Theorem~B]{ELMP26}}]
\label{thm:ELMP-pro-p}
Let $A_\Gamma$ be an Artin group and let $p$ be a prime number. The canonical epimorphism $A_\Gamma\to A_{\Gamma_p}$ induces an isomorphism
\[
\widehat{A_\Gamma}^{(p)}\cong \widehat{A_{\Gamma_p}}^{(p)}.
\]
Equivalently, $\widehat{A_\Gamma}^{(p)}$ has the pro-$p$ presentation
obtained by interpreting the defining presentation of
$A_{\Gamma_p}$ in the category of pro-$p$ groups.
\end{theorem}

Thus the pro-$p$ completion of $A_\Gamma$ depends only on the
$p$-part $\Gamma_p$. In passing from $\Gamma$ to $\Gamma_p$, every
odd-labeled component is collapsed, an even label $2m$ retains only
twice the $p$-part of $m$, and any parallel edges created by the
collapse are merged by keeping the smallest label. Consequently,
this passage may lose information about both the odd-labeled edges
and the original even labels.

Our first task is to show that, if $\widehat{A_\Gamma}^{(p)}$ is isomorphic to the pro-$p$ completion of a RAAG, then $\Gamma_p$ must itself be right-angled. For this purpose, we combine the retract structure of even Artin groups with a theorem on two-generated closed subgroups of
pro-$p$ RAAGs.

\begin{lemma}\label{lem:pro-p-completion-retract}
Let $p$ be a prime number, and let $H\leq G$ be a retract, i.e., the inclusion
$\iota: H\hookrightarrow G$ admits a retraction
$r: G\to H$ such that $r\circ\iota=\operatorname{id}_H$. Then the induced homomorphism
\[
\widehat\iota^{(p)}
\colon
\widehat H^{(p)}
\longrightarrow
\widehat G^{(p)}
\]
is a split injection with left inverse $\widehat r^{(p)}$.
Consequently, $\widehat H^{(p)}$ identifies with a closed retract of
$\widehat G^{(p)}$.
\end{lemma}

\begin{proof}
Since the pro-$p$ completion is functorial, 
\[
\widehat r^{(p)}\circ\widehat\iota^{(p)} = \widehat{r\circ\iota}^{(p)} = \operatorname{id}_{\widehat H^{(p)}}.
\]
Hence $\widehat\iota^{(p)}$ is injective. Its image is compact and
therefore closed in the Hausdorff group $\widehat G^{(p)}$, and the
restriction of $\widehat r^{(p)}$ to this image is its inverse.
\end{proof}

\bigskip
Recall that the free pro-$p$ group of rank $n$ is naturally isomorphic to the pro-$p$ completion $\widehat{F_n}^{(p)}$ of the abstract free group $F_n$ of rank $n$.

We will use the following result on closed topologically two-generated subgroups of pro-$p$ right-angled Artin groups.

\begin{theorem}[Casals-Ruiz--Pintonello--Zalesskii
{\cite[Theorem~3.14]{CPZ25}}]
\label{thm:two-generated-pro-p-raag}
Let $A_\Lambda$ be a right-angled Artin group. Every closed subgroup
of $\widehat{A_\Lambda}^{(p)}$ which is topologically generated by two
elements is either a free pro-$p$ group or a free abelian pro-$p$ group.
\end{theorem}
\begin{remark}
Theorem~\ref{thm:two-generated-pro-p-raag} is the special case of \cite[Theorem~3.14]{CPZ25} obtained by taking $\cC$ to be the class of finite $p$-groups. In this case, the pro-$\cC$ RAAG associated with $\Lambda$ is precisely $\widehat{A_{\Lambda}}^{(p)}$. Moreover, every closed subgroup of a pro-$p$ group is itself pro-$p$, and ``two-generated" in the pro-$p$ category means exactly topologically generated by two elements.
\end{remark}
Consequently, a topologically two-generated closed subgroup of $\widehat{A_\Lambda}^{(p)}$ cannot be both nonabelian and have nontrivial center. Indeed, suppose that such a subgroup $H$ exists. Since $H$ is nonabelian, it cannot be free abelian pro-$p$. Hence, by Theorem \ref{thm:two-generated-pro-p-raag}, $H$ must be a free pro-$p$ group. But a nonabelian free pro-$p$ group has trivial center
(see, for example, Ribes and Zalesskii~\cite[Corollary~8.7.3]{RZ10}),
a contradiction. In the next section, we construct precisely such an obstruction from every edge of label $2p^t$ with $t\geq 1$.

\section{Recognition of right-angled \texorpdfstring{$p$}{p}-parts}
\label{sec:recognition-p-parts}

We now show that right-angledness of the $p$-part of an Artin graph
is detected by the pro-$p$ completion of the associated Artin group.
We begin with the rank-two obstruction described above.

\subsection{A rank-two pro-\texorpdfstring{$p$}{p} obstruction}

Let $p$ be a prime number and let $t\geq 1$. We denote by
\[
D_{p^t}
=
\langle a,b\mid (ab)^{p^t}=(ba)^{p^t}\rangle
\]
the dihedral Artin group associated with a single edge of label
$2p^t$.

\begin{lemma}\label{lem:dihedral-pro-p-obstruction}
The pro-$p$ completion $\widehat{D_{p^t}}^{(p)}$ is topologically
two-generated, nonabelian, and has nontrivial center.
\end{lemma}

\begin{proof}
Set $s=ba$. Then $b=sa^{-1}$ and $ab=asa^{-1}$, so the defining relation is equivalent to
\[
as^{p^t}a^{-1}=s^{p^t}.
\]
Therefore,
\[
D_{p^t}\cong\langle a,s\mid [a,s^{p^t}]=1\rangle,
\]
and the images of $a$ and $s$ in the canonical homomorphism $i:D_{p^t}\to\widehat{D_{p^t}}^{(p)}$ topologically generate the completion. In particular, $z=s^{p^t}$ is central in $D_{p^t}$.

We first show that the image of $z$ in
$\widehat{D_{p^t}}^{(p)}$ is nontrivial. Let
$C_{p^{t+1}}=\langle c\rangle$ be the cyclic group of order $p^{t+1}$. Consider the homomorphism $f:D_{p^t}\to C_{p^{t+1}}$ defined by taking $f(a)=1$ and $f(s)=c$. Then $f(z)=c^{p^t}\neq 1$. Thus $z$ survives in a
finite $p$-group quotient and has nontrivial image $\overline z$ in
$\widehat{D_{p^t}}^{(p)}$. Since $z \in Z(D_{p^t})$, for every $g \in D_{p^t}$ we have $[\bar z,i(g)]=i([z,g])=1$. Hence $i(D_{p^t})\subseteq C_{\widehat{D_{p^t}}^{(p)}}(\bar z).$

The centralizer $C_{\widehat{D_{p^t}}^{(p)}}(\bar z)$ is closed, since it is
the inverse image of $\{1\}$ under the continuous map $\widehat{D_{p^t}}^{(p)}
\to\widehat{D_{p^t}}^{(p)}$ defined by $x\mapsto [x,\bar z].$ Since $i(D_{p^t})$ is dense in $\widehat{D_{p^t}}^{(p)}$, 
$C_{\widehat{D_{p^t}}^{(p)}}(\bar z)=\widehat{D_{p^t}}^{(p)}$.
Therefore, $\bar z \in Z(\widehat{D_{p^t}}^{(p)})$.

In order to show that the completion is nonabelian, consider the finite
$p$-group
\[
U_3(\mathbb F_p)
=
\left\{
\begin{pmatrix}
1 & x & u\\
0 & 1 & y\\
0 & 0 & 1
\end{pmatrix}
\ \middle|\
x,y,u\in\mathbb F_p
\right\}.
\]
Let $X=I+E_{12}$ and $Y=I+E_{23}$. Then $X^p=Y^p=I$ but
$[X,Y]=I+E_{13}\neq I$. Since $Y^{p^t}=1$, the epimorphism $\varphi: D_{p^t}\to U_3(\mathbb{F}_p)$ defined by $\varphi(a)=X$ and $\varphi(s)=Y$ has nonabelian image. Hence, by universal property, $\widehat{D_{p^t}}^{(p)}$ is nonabelian. 
\end{proof}

\begin{remark}\label{rem:no-residual-p-needed}
The group $D_{p^t}\cong BS(p^t,p^t)$ is residually $p$;
see, for example, Moldavanskii~\cite[Theorem~2]{Mol18}.
Together with the nonabelianity and nontrivial center of $D_{p^t}$, this gives an alternative proof of Lemma~\ref{lem:dihedral-pro-p-obstruction}.
The proof above uses only two explicit finite $p$-group quotients.
\end{remark}

\subsection{Pro-$p$ recognition of right-angled $p$-parts}

\begin{theorem}[\texorpdfstring{Pro-$p$ recognition of right-angled $p$-parts}{Pro-p recognition of right-angled p-parts}]
\label{thm:pro-p-recognition-p-parts}
Let $A_\Gamma$ and $A_\Delta$ be Artin groups, and let $p$ be a
prime number. If $\widehat{A_\Gamma}^{(p)}\cong\widehat{A_\Delta}^{(p)}$, then $\Gamma_p$ is right-angled if and only if $\Delta_p$ is right-angled. Moreover, if these equivalent conditions hold, then $\Gamma_p\cong\Delta_p$ as labeled graphs.
\end{theorem}

\begin{proof}
By the theorem of Escart\'in Ferrer--Leoni--Mart\'inez P\'erez, stated as Theorem~\ref{thm:ELMP-pro-p}, we have
\[
\widehat{A_{\Gamma_p}}^{(p)}
\cong
\widehat{A_\Gamma}^{(p)}
\cong
\widehat{A_\Delta}^{(p)}
\cong
\widehat{A_{\Delta_p}}^{(p)}.
\]

Thus, we only need to show that if $\Gamma_p$ is right-angled, then $\Delta_p$ must also be right-angled. Suppose, towards a contradiction, that $\Delta_p$ is not
right-angled. Then it must contain at least one edge $e\in E(\Delta_p)$ labeled by $2p^t$ for some $t\geq 1$. The corresponding rank-two standard parabolic subgroup $A_e$ is isomorphic to $D_{p^t}$. Since $A_{\Delta_p}$ is even,
by Lemma~\ref{lem:standard-retract-even}, $A_e$ is a
retract of $A_{\Delta_p}$. Then we have a closed embedding
\[
\widehat{D_{p^t}}^{(p)}
\hookrightarrow
\widehat{A_{\Delta_p}}^{(p)}
\cong
\widehat{A_{\Gamma_p}}^{(p)}.
\]
By Lemma~\ref{lem:dihedral-pro-p-obstruction}, $\widehat{D_{p^t}}^{(p)}$ is topologically two-generated, nonabelian, and has nontrivial center. It is therefore neither free abelian pro-$p$ nor free pro-$p$, contradicting Theorem~\ref{thm:two-generated-pro-p-raag}. Therefore, $\Delta_p$ is right-angled. 

When both $p$-parts are right-angled, by Theorem~\ref{thm:KW-pro-p-rigidity}, we have $\Gamma_p\cong\Delta_p$.
\end{proof}

\begin{corollary}\label{cor:recover-all-p-parts}
Let $A_\Gamma$ be an arbitrary Artin group and let $A_\Lambda$ be a
right-angled Artin group. If
$\widehat{A_\Gamma}\cong\widehat{A_\Lambda}$, then $\Gamma_p\cong\Lambda$
for every prime $p$.
\end{corollary}

\begin{proof}
If $\widehat{A_\Gamma}\cong\widehat{A_\Lambda}$, then $\widehat{A_\Gamma}^{(p)}\cong\widehat{A_\Lambda}^{(p)}$. By construction, $\Lambda_p=\Lambda$ that is right-angled. So by Theorem~\ref{thm:pro-p-recognition-p-parts} , we have $\Gamma_p\cong\Lambda$ for all $p$.
\end{proof}

Theorem~\ref{thm:pro-p-recognition-p-parts} also yields two relative
rigidity statements which require only pro-$p$ information. For a
prime number $p$, we say that a graph $\Gamma$ is
\emph{$p$-primary} if each edge of $\Gamma$ has label of the form
$2p^t$ for some $t\geq 0$.

\begin{corollary}[Pro-$p$ rigidity of $p$-primary Artin groups relative to RAAGs]
\label{cor:p-primary-relative-pro-p-rigidity}
Let $p$ be a prime number, let $A_\Gamma$ be an Artin group whose defining
graph $\Gamma$ is $p$-primary, and let $A_\Lambda$ be a right-angled
Artin group. If
$\widehat{A_\Gamma}^{(p)}\cong\widehat{A_\Lambda}^{(p)},$
then $\Gamma$ is right-angled and $\Gamma\cong\Lambda$. In
particular, $A_\Gamma\cong A_\Lambda$.
\end{corollary}

\begin{proof}
Since every edge label of $\Gamma$ has the form $2p^t$, no vertices
are identified and no edge labels are changed in the construction of
$\Gamma_p$. So $\Gamma_p=\Gamma$. Since $\Lambda$ is right-angled,
we also have $\Lambda_p=\Lambda$. The conclusion then follows from Theorem~\ref{thm:pro-p-recognition-p-parts}.
\end{proof}

\begin{corollary}[Relative rigidity from pro-$p$ completions]
\label{cor:even-all-pro-p-rigidity}
Let $A_\Gamma$ be an arbitrary even Artin group and let $A_\Lambda$ be a
right-angled Artin group. If $\widehat{A_\Gamma}^{(p)}\cong\widehat{A_\Lambda}^{(p)}$
for every prime $p$, then $\Gamma$ is right-angled and
$\Gamma\cong\Lambda$. In particular,
$A_\Gamma\cong A_\Lambda$.
\end{corollary}

\begin{proof}
Since $\Lambda_p=\Lambda$ for every prime $p$, by Theorem~\ref{thm:pro-p-recognition-p-parts}, $\Gamma_p\cong\Lambda$ for every prime $p$, i.e., every $\Gamma_p$ is right-angled.

Now suppose that there is an edge $e\in E(\Gamma)$ labelled by $2m>2$. Choose a prime
$p|m$ such that $m=kp^t$ with $t\ge 1$ and $\gcd(k,p)=1$. Since $\Gamma$ is even, no vertices will be identified in the construction of $\Gamma_p$. Since $\Gamma$ is simplicial, no multiple edges arise. Therefore, $e$ gives rise to an edge of $\Gamma_p$ labeled by $2p^t>2$, contradicting that $\Gamma_p$ is right-angled.

Therefore, every edge of $\Gamma$ has label $2$, and $\Gamma$ is right-angled. Consequently, $\Gamma_p=\Gamma$ for every prime $p$ and we have $\Gamma\cong\Lambda$.
\end{proof}

\begin{remark}\label{rem:finitely-many-primes-suffice}
In Corollary~\ref{cor:even-all-pro-p-rigidity}, it is not necessary to consider every prime. Actually, it suffices to choose a set of primes $S$ such that, for every edge label $2m>2$ of $\Gamma$, some prime $p\in S$ divides $m$, and to assume that
\[
\widehat{A_\Gamma}^{(p)}
\cong
\widehat{A_\Lambda}^{(p)}
\]
for every $p\in S$. Since $\Gamma$ has finitely many edges, the primes
in $S$ needed to detect its edge labels may be chosen to form a
finite set.
\end{remark}

\section{Detecting odd-labeled edges}
\label{sec:detecting-odd-edges}

The preceding section recovers the graph obtained from $\Gamma$ by
collapsing its odd-labeled components. Under the assumption that
$\widehat{A_\Gamma}\cong\widehat{A_\Lambda}$, where $A_\Lambda$ is
right-angled, we now use finite quotients to show that $\Gamma$ cannot
contain any odd-labeled edges.

\subsection{The odd-collapse graph}

Define an equivalence relation $\sim_{\mathrm{odd}}$ on
$V(\Gamma)$ by setting $u\sim_{\mathrm{odd}}v$ if there exists a path from $u$ to $v$ all of whose edge labels are odd. Equivalently, $\sim_{\mathrm{odd}}$ is the equivalence relation generated by the pairs of endpoints of odd-labeled edges.

For a group $G$, denote its abelianization by
\[
G^{\mathrm{ab}}=G/[G,G],
\]
where $[G,G]$ is the commutator subgroup generated by $\{g^{-1}h^{-1}gh:g,h\in G\}$.
\begin{lemma}\label{lem:artin-abelianization}
Let $A_\Gamma$ be an Artin group. Then
\[
A_\Gamma^{\mathrm{ab}}\cong\ZZ^{|V(\Gamma)/{\sim_{\mathrm{odd}}}|}.
\]
More precisely, two standard generators have the same image in
$A_\Gamma^{\mathrm{ab}}$ if and only if they are in the same
$\sim_{\mathrm{odd}}$-equivalence class. In particular,
\[
\operatorname{rank} A_\Gamma^{\mathrm{ab}}
=
|V(\Gamma_{\mathrm{odd}})|.
\]
\end{lemma}

\begin{proof}
In the abelianization, an Artin relation associated with an edge
$\{u,v\}\in E(\Gamma)$ of even label $2k$ becomes tautological, since both sides
contain $k$ copies of $u$ and $k$ copies of $v$. If the edge has odd
label $2k+1$, the corresponding relation becomes
\[
(k+1)u+kv=ku+(k+1)v,
\]
so $u=v$. Thus
\[
A_\Gamma^{\mathrm{ab}}\cong\ZZ^{|V(\Gamma)|}
/
\langle u-v\mid \{u,v\}\in E(\Gamma)\text{ has odd label}\rangle.
\]
The quotient identifies precisely the vertices lying in the same
$\sim_{\mathrm{odd}}$-equivalence class. Therefore, 
\[
A_\Gamma^{\mathrm{ab}}\cong\ZZ^{|V(\Gamma)/{\sim_{\mathrm{odd}}}|}.
\]

\end{proof}

\begin{definition}\label{def:odd-collapse}
The \emph{odd-collapse graph} $\Gamma_{\mathrm{odd}}$ is defined as the finite simplicial graph whose vertices are the $\sim_{\mathrm{odd}}$-equivalence classes. Two distinct classes
$C$ and $D$ are connected by an edge if $\Gamma$ contains an edge with
one endpoint in $C$ and the other in $D$. We then discard loops, merge parallel edges, and assign label $2$ to every edge of
$\Gamma_{\mathrm{odd}}$, so that $\Gamma_{\mathrm{odd}}$ is
right-angled.
\end{definition}

\begin{example}\label{ex:odd-collapse}
Figure~\ref{fig:odd-collapse} illustrates the construction of
$\Gamma_{\mathrm{odd}}$. The odd-labeled components of $\Gamma$ are exactly
\[
\{a,b,c\},\qquad \{d,e\},\qquad \{f,g\},\qquad \{h\},
\]
so the vertices of $\Gamma_{\mathrm{odd}}$ are as follows:
\[
x_1=[a]_{\mathrm{odd}}=[b]_{\mathrm{odd}}=[c]_{\mathrm{odd}},\quad
x_2=[d]_{\mathrm{odd}}=[e]_{\mathrm{odd}},\quad
x_3=[f]_{\mathrm{odd}}=[g]_{\mathrm{odd}},\quad
x_4=[h]_{\mathrm{odd}}.
\]
\end{example}

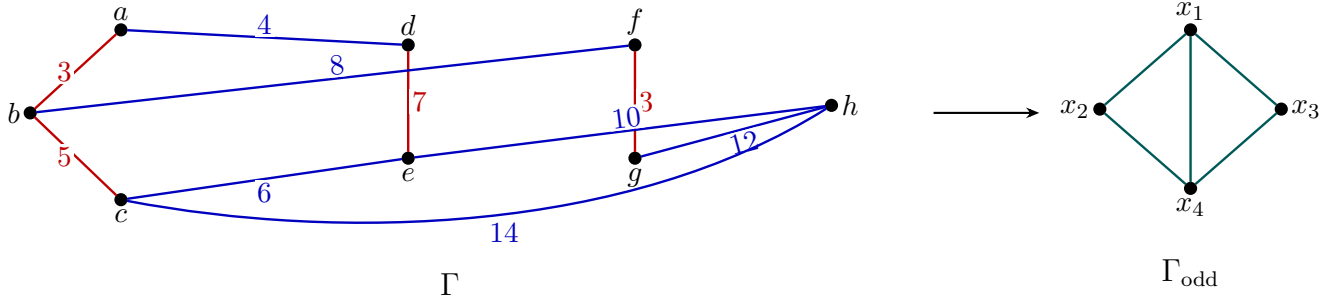
\begin{figure}[ht]
\centering
\begin{tikzpicture}[
    x=1cm,
    y=1cm,
    >={Stealth[length=5pt,width=4pt]},
    vertex/.style={
        circle,
        fill=black,
        inner sep=1.75pt
    },
    oddedge/.style={
        red!75!black,
        line width=0.9pt,
        line cap=round
    },
    evenedge/.style={
        blue!75!black,
        line width=0.9pt,
        line cap=round
    },
    collapseedge/.style={
        teal!70!black,
        line width=0.9pt,
        line cap=round
    },
    lab/.style={
        fill=white,
        inner sep=1pt,
        font=\small
    },
    every label/.style={
        font=\small,
        inner sep=1pt
    }
]

\begin{scope}

    \node[vertex,label=above:$a$] (a) at (0,2.25) {};
    \node[vertex,label=left:$b$]  (b) at (-1.2,1.15) {};
    \node[vertex,label=below:$c$] (c) at (0,0) {};

    \node[vertex,label=above:$d$] (d) at (3.8,2.05) {};
    \node[vertex,label=below:$e$] (e) at (3.8,0.55) {};

    \node[vertex,label=above:$f$] (f) at (6.8,2.05) {};
    \node[vertex,label=below:$g$] (g) at (6.8,0.55) {};

    \node[vertex,label=right:$h$] (h) at (9.4,1.25) {};

    \draw[oddedge]
        (a) -- (b)
        node[midway,left,lab] {$3$};

    \draw[oddedge]
        (b) -- (c)
        node[midway,left,lab] {$5$};

    \draw[oddedge]
        (d) -- (e)
        node[midway,right,lab] {$7$};

    \draw[oddedge]
        (f) -- (g)
        node[midway,right,lab] {$3$};

    \draw[evenedge]
        (a) -- (d)
        node[pos=.50,above,lab] {$4$};

    \draw[evenedge]
        (c) -- (e)
        node[pos=.50,below,lab] {$6$};

    \draw[evenedge]
        (b) -- (f)
        node[pos=.51,above,sloped,lab] {$8$};

    \draw[evenedge]
        (e) -- (h)
        node[pos=.52,above,sloped,lab] {$10$};

    \draw[evenedge]
        (g) -- (h)
        node[pos=.54,below,sloped,lab] {$12$};

    \draw[evenedge]
        (c)
        .. controls (2.9,-0.55) and (6.7,-0.45) ..
        node[pos=.53,below=2pt,lab] {$14$}
        (h);

    \node[font=\normalsize] at (4.35,-1.12) {$\Gamma$};

\end{scope}

\draw[->,line width=0.8pt]
    (10.75,1.15) -- (12.15,1.15);

\begin{scope}[shift={(14.15,0.15)}]

    \node[vertex,label=above:$x_1$] (x1) at (0,2.10) {};
    \node[vertex,label=left:$x_2$]  (x2) at (-1.20,1.05) {};
    \node[vertex,label=right:$x_3$] (x3) at (1.20,1.05) {};
    \node[vertex,label=below:$x_4$] (x4) at (0,0) {};

    \draw[collapseedge] (x1) -- (x2);
    \draw[collapseedge] (x1) -- (x3);
    \draw[collapseedge] (x1) -- (x4);
    \draw[collapseedge] (x2) -- (x4);
    \draw[collapseedge] (x3) -- (x4);

    \node[font=\normalsize] at (0,-1.12) {$\Gamma_{\mathrm{odd}}$};

\end{scope}

\end{tikzpicture}

\caption{An example of the odd-collapse graph.}
\label{fig:odd-collapse}
\end{figure}

\begin{proposition}\label{prop:odd-collapse-properties}
The odd-collapse graph has the following properties.
\begin{enumerate}
\item For every prime $p$, the underlying simplicial graph of
$\Gamma_p$ is canonically isomorphic to
$\Gamma_{\mathrm{odd}}$. If $\Gamma_p$ is right-angled, then this is
an isomorphism of labeled graphs.

\item The assignment $v\mapsto[v]_{\mathrm{odd}}$ on the standard
generators induces a canonical epimorphism
\[
\pi_\Gamma:A_\Gamma\twoheadrightarrow A_{\Gamma_{\mathrm{odd}}}.
\]
\end{enumerate}
\end{proposition}

\begin{proof}
\begin{enumerate}

\item[(1)]
Let
\[
q: V(\Gamma)\to V(\Gamma)/{\sim_{\mathrm{odd}}}=V(\Gamma_{\mathrm{odd}})
\]
be the quotient map. The first step in the construction of
$\Gamma_p$ identifies exactly the vertices in each fiber of $q$.
Thus the vertex set $V(\Gamma_p)$ is naturally identified with
$V(\Gamma_{\mathrm{odd}})$.

Let $C$ and $D$ be distinct equivalence classes. After the first
step, there is an edge joining $C$ and $D$ if and only if $\Gamma$
contains an edge with one endpoint in $C$ and the other in $D$. The
second step changes only edge labels, while the third replaces a
nonempty collection of parallel edges by a single edge. Neither step
changes adjacency. Hence $C$ and $D$ are adjacent in $\Gamma_p$ if
and only if they are adjacent in $\Gamma_{\mathrm{odd}}$. This gives
a canonical isomorphism between the underlying simplicial graphs of
$\Gamma_p$ and $\Gamma_{\mathrm{odd}}$.

If $\Gamma_p$ is right-angled, then every edge of $\Gamma_p$ has
label $2$. Since every edge of $\Gamma_{\mathrm{odd}}$ is also
assigned label $2$, this canonical isomorphism preserves labels.

\item[(2)]
Consider the assignment on the standard generators $v\mapsto [v]_{\mathrm{odd}}$ for $v\in V(\Gamma)$. Let $\{v,w\}\in E(\Gamma)$ be labeled by $m$. If
$v\sim_{\mathrm{odd}}w$, then
$[v]_{\mathrm{odd}}=[w]_{\mathrm{odd}}$, and both alternating words
of length $m$ map to $[v]_{\mathrm{odd}}^m.$

Now suppose that $v\not\sim_{\mathrm{odd}}w$. Then $m=2k$ for some $k\in\NN$. By construction, $[v]_{\mathrm{odd}}$ and
$[w]_{\mathrm{odd}}$ are joined by an edge in $\Gamma_{\mathrm{odd}}$, and
so
\[
[v]_{\mathrm{odd}}[w]_{\mathrm{odd}}=[w]_{\mathrm{odd}}[v]_{\mathrm{odd}}\in A_{\Gamma_{\mathrm{odd}}}.
\]
Hence both alternating words of length $2k$ map to
$[v]_{\mathrm{odd}}^k[w]_{\mathrm{odd}}^k.$
Therefore, every defining relation of $A_\Gamma$ is preserved and the
assignment induces a homomorphism
\[
\pi_\Gamma\colon
A_\Gamma\longrightarrow A_{\Gamma_{\mathrm{odd}}}.
\]
It is surjective because every standard generator of
$A_{\Gamma_{\mathrm{odd}}}$ is of the form
$[v]_{\mathrm{odd}}$ for some $v\in V(\Gamma)$.

\end{enumerate}
\end{proof}
\begin{remark}
The canonical epimorphism
$\pi_\Gamma:A_\Gamma\twoheadrightarrow A_{\Gamma_{\mathrm{odd}}}$
induces an isomorphism on abelianizations. Indeed,
\[
A_\Gamma^{\mathrm{ab}}
\cong
\ZZ^{|V(\Gamma_{\mathrm{odd}})|}
\cong
A_{\Gamma_{\mathrm{odd}}}^{\mathrm{ab}}.
\]
Thus $\ker(\pi_\Gamma)\subseteq[A_\Gamma,A_\Gamma]$.
The graph $\Gamma_{\mathrm{odd}}$ may therefore be viewed as an
enhancement of the abelianization: its vertices record exactly the
identifications of standard generators forced in the abelianization,
while its edges retain the adjacency information between the
resulting odd components.
\end{remark}

The preceding proposition identifies the underlying simplicial graph
of $\Gamma_p$. We now record explicitly how the edge labels of
$\Gamma_p$ are determined by the labels of $\Gamma$.

Let $C$ and $D$ be distinct adjacent vertices of
$\Gamma_{\mathrm{odd}}$, viewed as $\sim_{\mathrm{odd}}$-equivalence
classes in $V(\Gamma)$. Every edge of $\Gamma$ joining a vertex of
$C$ to a vertex of $D$ has even label and assume that these edges have labels
\[
2m_1,\ldots,2m_r,
\]
for some $m_1,\cdots,m_r\in\NN$. Now define
\[
d_\Gamma(C,D)
\coloneq
\gcd(m_1,\ldots,m_r).
\]

Given a prime number $p$ and a positive integer $m$, denote by
$v_p(m)$ the $p$-adic valuation of $m$, that is,
\[
v_p(m)=\max\{t\geq 0 : p^t\mid m\}.
\]
Equivalently, if $m=kp^t$ with $\gcd(k,p)=1$, then
$v_p(m)=t$.

Now we write $m_i=k_ip^{t_i}$ where $t_i=v_p(m_i)$ and $\gcd(k_i,p)=1$ for $1\le i\le r$. In the construction of $\Gamma_p$, the edge of $\Gamma$ labelled $2m_i=2k_ip^{t_i}$ is replaced by an edge labeled $2p^{t_i}$. After the odd-labeled components are collapsed, all of these edges join the same pair of vertices $C$ and $D$. The parallel edges are
then merged, retaining the smallest label in $\Gamma_p$. Therefore, the resulting edge
of $\Gamma_p$ has label
\[
\min_{1\leq i\leq r} 2p^{t_i}=2p^{\min \{t_i\}}=2p^{\min_i v_p(m_i)}=2p^{v_p(d_\Gamma(C,D))}.
\]

\begin{corollary}[Classification from all pro-$p$ completions]
\label{cor:all-pro-p-classification}
Let $A_\Gamma$ be an Artin group and let $A_\Lambda$ be a
right-angled Artin group. Then the following are equivalent:
\begin{enumerate}
\item $\widehat{A_\Gamma}^{(p)}\cong\widehat{A_\Lambda}^{(p)}$ for every prime $p$;

\item $\Gamma_p\cong\Lambda$ as labeled graphs for every prime $p$;

\item $\Gamma_{\mathrm{odd}}\cong\Lambda$ and $d_\Gamma(C,D)=1$ for every edge $\{C,D\}\in E(\Gamma_{\mathrm{odd}})$.
\end{enumerate}
\end{corollary}

\begin{proof}
$(1)\Longrightarrow (2)$: Suppose first that {\rm (1)} holds. Since $\Lambda$ is right-angled,
we have $\Lambda_p=\Lambda$ for every prime $p$. Hence by Theorem~\ref{thm:pro-p-recognition-p-parts}, we have
\[
\Gamma_p\cong\Lambda
\]
for every prime $p$. Thus {\rm (2)} holds.

$(2)\Longrightarrow (3)$: Now suppose that {\rm (2)} holds. Since $\Gamma_p$ is right-angled for every prime $p$, by
Proposition~\ref{prop:odd-collapse-properties}, the underlying
simplicial graph of $\Gamma_p$ is canonically isomorphic to
$\Gamma_{\mathrm{odd}}$. Therefore, $\Gamma_{\mathrm{odd}}\cong\Lambda$.

Let $\{C,D\}\in E(\Gamma_{\mathrm{odd}})$. Since the corresponding edge of $\Gamma_p$ has label 
\[
2p^{v_p(d_\Gamma(C,D))}=2,
\]
we have that
\[
v_p(d_\Gamma(C,D))=0
\]
for every prime $p$, i.e., $d_\Gamma(C,D)$ is not divisible by any prime numbers. Thus,
\[
d_\Gamma(C,D)=1.
\]

$(3)\Longrightarrow (1)$: Finally, if {\rm (3)} holds. Then every edge of $\Gamma_p$ has label $2$. Hence,
\[
\Gamma_p\cong\Gamma_{\mathrm{odd}}\cong\Lambda
\]
for every prime $p$. By Theorem~\ref{thm:ELMP-pro-p},
\[
\widehat{A_\Gamma}^{(p)}
\cong
\widehat{A_{\Gamma_p}}^{(p)}
\cong
\widehat{A_\Lambda}^{(p)}
\]
for every prime $p$. 
\end{proof}

The following examples illustrate a limitation of the information
detected by the collection of all pro-$p$ completions.

The simplest example is the braid group
\[
B_3=A_{I_2(3)}.
\]
Its odd-collapse graph consists of a single vertex, and hence Corollary~\ref{cor:all-pro-p-classification} gives
\[
\widehat{B_3}^{(p)}
\cong
\widehat{\ZZ}^{(p)}
\cong
\ZZ_p
\]
for every prime $p$. Nevertheless,
\[
\widehat{B_3}\not\cong\widehat{\ZZ},
\]
as $B_3$ admits the nonabelian finite quotient $S_3$, whereas every finite quotient of $\ZZ$ is cyclic.

A slightly less degenerate example is provided by the labeled triangle
with edge labels $3$, $4$, and $6$.

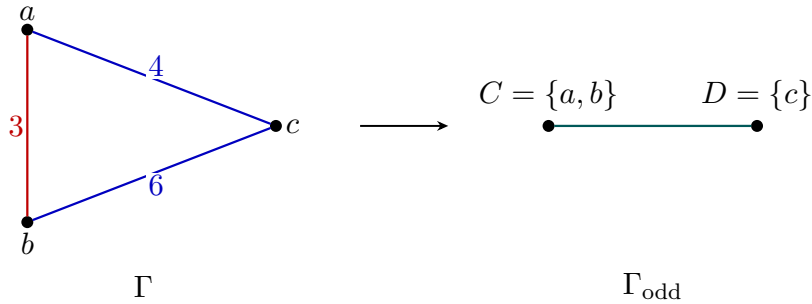
\begin{figure}[ht]
\centering
\resizebox{0.68\textwidth}{!}{%
\begin{tikzpicture}[
    >={Stealth[length=4pt,width=3.5pt]},
    font=\small,
    vertex/.style={
        circle,
        fill=black,
        inner sep=1.5pt
    },
    oddedge/.style={
        red!75!black,
        line width=0.8pt,
        line cap=round
    },
    evenedge/.style={
        blue!75!black,
        line width=0.8pt,
        line cap=round
    },
    collapseedge/.style={
        teal!70!black,
        line width=0.8pt,
        line cap=round
    },
    lab/.style={
        fill=white,
        inner sep=0.8pt,
        font=\small
    },
    every label/.style={
        font=\small,
        inner sep=1pt
    }
]

\begin{scope}

    \node[vertex,label=above:$a$] (a) at (0,2.4) {};
    \node[vertex,label=below:$b$] (b) at (0,0) {};
    \node[vertex,label=right:$c$] (c) at (3.1,1.2) {};

    \draw[oddedge]
        (a) -- (b)
        node[midway,left,lab] {$3$};

    \draw[evenedge]
        (a) -- (c)
        node[pos=.52,above,lab] {$4$};

    \draw[evenedge]
        (b) -- (c)
        node[pos=.52,below,lab] {$6$};

    \node at (1.45,-0.8) {$\Gamma$};

\end{scope}

\draw[->,line width=0.7pt]
    (4.15,1.2) -- (5.25,1.2);

\begin{scope}[shift={(6.5,0)}]

    \node[
        vertex,
        label={[label distance=2pt]above:$C=\{a,b\}$}
    ] (C) at (0,1.2) {};

    \node[
        vertex,
        label={[label distance=2pt]above:$D=\{c\}$}
    ] (D) at (2.6,1.2) {};

    \draw[collapseedge] (C) -- (D);

    \node at (1.3,-0.8) {$\Gamma_{\mathrm{odd}}$};

\end{scope}

\end{tikzpicture}%
}
\caption{The odd-collapse graph for the labeled triangle with edge
labels $3$, $4$, and $6$.}
\label{fig:346-odd-collapse}
\end{figure}

Here the odd components are
\[
C=\{a,b\},
\qquad
D=\{c\},
\]
and the edges joining them have half-labels $2$ and $3$. Hence
\[
d_\Gamma(C,D)=\gcd(2,3)=1.
\]
Since $\Gamma_{\mathrm{odd}}$ is a single edge, by Corollary~\ref{cor:all-pro-p-classification}, we have
\[
\widehat{A_\Gamma}^{(p)}
\cong
\widehat{\ZZ^2}^{(p)}
\]
for every prime $p$.

On the other hand,
\[
\widehat{A_\Gamma}\not\cong\widehat{\ZZ^2}.
\]
Indeed, the assignment
\[
a\mapsto(12),\qquad
b\mapsto(23),\qquad
c\mapsto1
\]
defines a surjection $A_\Gamma\twoheadrightarrow S_3$, whereas every
finite quotient of $\ZZ^2$ is abelian.

These examples show that the collection of all finite $p$-group
quotients, even when all primes are considered simultaneously, does
not in general detect the odd-labeled structure of an Artin graph.
Thus finite quotients beyond finite $p$-groups are needed to recover
the missing information.

\subsection{A finite-quotient obstruction}

\begin{lemma}[Finite-quotient obstruction]
\label{lem:finite-quotient-obstruction}
Let $G$ and $H$ be finitely generated groups, and let
$\pi\colon G\twoheadrightarrow H$ be an epimorphism. If
$\widehat G\cong\widehat H$, then every homomorphism from $G$ to a
finite group factors through $\pi$. Equivalently, $\ker(\pi)\subseteq\ker(i:G\to\widehat{G})$.
\end{lemma}

\begin{proof}
Let $Q$ be a finite group. Since $\pi$ is surjective, the pullback
\[
\begin{aligned}
\pi^*:\Hom(H,Q)
&\longrightarrow
\Hom(G,Q),\\
\psi
&\longmapsto
\psi\circ\pi
\end{aligned}
\]
is injective. Since $G$ and $H$ are finitely generated, $\Hom(G,Q)$ and $\Hom(H,Q)$ are finite and by Proposition~\ref{prop:hom-finite-quotients}, the isomorphism
$\widehat G\cong\widehat H$ implies
\[
|\Hom(G,Q)|=|\Hom(H,Q)|.
\]
Hence $\pi^*$ is also surjective, and every homomorphism $G\to Q$ factors
through $\pi$.

Since this holds for every finite group $Q$, each element of
$\ker(\pi)$ lies in the kernel of every homomorphism from $G$ to a finite group. Therefore, $\ker(\pi)\subseteq\ker(i:G\to\widehat{G})$.
\end{proof}

\begin{remark}\label{rem:no-compatibility}
The isomorphism $\widehat G\cong\widehat H$ in
Lemma~\ref{lem:finite-quotient-obstruction} need not be induced by, or compatible with, $\pi$. The proof uses only the cardinalities of the Hom sets.
\end{remark}

\begin{remark}\label{rem:hopfian-type}
Lemma~\ref{lem:finite-quotient-obstruction} may be viewed as a Hopfian-type statement. If $G$ is residually finite, then the canonical homomorphism $i:G\to\widehat G$ is injective, and hence the lemma implies that any epimorphism $\pi:G\twoheadrightarrow H$ between finitely generated groups satisfying
$\widehat G\cong\widehat H$ is an isomorphism. In particular, taking $H=G$ recovers the classical fact that every finitely generated residually finite group is Hopfian.

As we've mentioned before, for arbitrary Artin groups, residual finiteness is still not known in general. Thus, in the present setting, the useful conclusion of the lemma is the weaker statement
\[
\ker(\pi)\subseteq\ker(G\to\widehat G),
\]
namely that every element of $\ker(\pi)$ is invisible in all finite
quotients of $G$.
\end{remark}

\begin{proposition}\label{prop:no-odd-edges}
Let $A_\Gamma$ be an arbitrary Artin group and let $A_\Lambda$ be a right-angled Artin group. If $\widehat{A_\Gamma}\cong\widehat{A_\Lambda}$, then $\Gamma$ has no odd-labeled edges.
\end{proposition}

\begin{proof}
By Corollary~\ref{cor:all-pro-p-classification},
$\Gamma_{\mathrm{odd}}\cong\Lambda$, so
$A_{\Gamma_{\mathrm{odd}}}\cong A_\Lambda$, and thus
\[
\widehat{A_\Gamma}
\cong
\widehat{A_\Lambda}
\cong
\widehat{A_{\Gamma_{\mathrm{odd}}}}.
\]
Let  $\pi_\Gamma\colon
A_\Gamma\twoheadrightarrow A_{\Gamma_{\mathrm{odd}}}$ be the canonical epimorphism  from Proposition~\ref{prop:odd-collapse-properties}. Assume that there exists an odd-labeled edge $\{u,v\}\in E(\Gamma)$. Then
\[
\pi_\Gamma(u)=\pi_\Gamma(v). 
\]

On the other hand, let
$q_\Gamma\colon A_\Gamma\twoheadrightarrow W_\Gamma$ be the canonical quotient map to the corresponding Coxeter group. If the label of $\{u,v\}$ is $m$, then the standard rank-two parabolic subgroup generated by the images of $u$ and $v$ is the dihedral Coxeter group of order $2m$; see \cite[Theorem~4.1.6]{Dav08}. In particular, the images of $u$ and $v$ in $W_\Gamma$ are distinct.

Since Coxeter groups are residually finite, there exist a finite group $Q$ and a homomorphism $\rho\colon W_\Gamma\to Q$ such that
\[
\rho(q_\Gamma(u))
\neq
\rho(q_\Gamma(v)).
\]
Thus the composite $\varphi=\rho\circ q_\Gamma$ is a homomorphism $A_\Gamma\to Q$ that does not factor through $\pi_\Gamma$,
a contradiction. Therefore, $\Gamma$ has no odd-labeled edges.
\end{proof}

\section{Relative profinite rigidity}
\label{sec:relative-profinite-rigidity}

We can now prove the main theorem.
\begin{theorem}\label{thm:profinite-rigidity}
Let $A_\Gamma$ be an Artin group and let $A_\Lambda$ be a right-angled Artin group. If $\widehat{A_\Gamma}\cong\widehat{A_\Lambda},$
then $\Gamma$ is right-angled and $\Gamma\cong\Lambda$. In particular, $A_\Gamma\cong A_\Lambda$. Consequently, right-angled Artin groups are profinitely rigid relative to the class of Artin groups.
\end{theorem}

\begin{proof}
By Proposition~\ref{prop:no-odd-edges}, $\Gamma$ has no odd-labeled edges, and hence $A_\Gamma$ is even. Moreover, by Lemma~\ref{lem:maximal-pro-p-quotient}, we have
\[
\widehat{A_\Gamma}^{(p)}
\cong
\widehat{A_\Lambda}^{(p)}
\]
for every prime $p$. Then Corollary~\ref{cor:even-all-pro-p-rigidity}
gives $\Gamma\cong\Lambda,$ and in particular $\Gamma$ is right-angled. Hence $A_\Gamma\cong A_\Lambda$.

The final assertion follows by interchanging the two Artin groups if necessary.
\end{proof}

Thus the profinite completion not only detects right-angledness within
the class of Artin groups, but also determines the defining graph.
In particular, if $G$ is a finitely generated group satisfying
\[
\widehat G\cong\widehat{A_\Lambda}
\]
for some RAAG $A_\Lambda$, but $G\not\cong A_\Lambda$, then $G$ is
not an Artin group.

\begin{remark}[Grothendieck rigidity and absolute flexibility]
\label{rem:absolute-flexibility}
The conclusion of Theorem~\ref{thm:profinite-rigidity} is necessarily
relative. As recalled in the introduction, Platonov--Tavgen'
~\cite{PT86} considered an epimorphism
\[
q:F_4\twoheadrightarrow Q,
\]
where $Q$ is Higman's group
\[
Q=\left\langle a,b,c,d\ \middle| \
a^{-1}ba=b^2,\;
b^{-1}cb=c^2,\;
c^{-1}dc=d^2,\;
d^{-1}ad=a^2
\right\rangle,
\]
and formed the associated fibre product
\[
P=\{(x,y)\in F_4\times F_4 : q(x)=q(y)\}.
\]
The group $Q$ has no nontrivial finite quotients and satisfies
$H_2(Q,\ZZ)=0$, and the inclusion
\[
P\hookrightarrow F_4\times F_4
\]
induces an isomorphism
\[
\widehat P\cong\widehat{F_4\times F_4}.
\]
Moreover, $P$ is finitely generated and is not isomorphic to
$F_4\times F_4$.

Since
\[
A_{K_{4,4}}\cong F_4\times F_4,
\]
where $K_{4,4}$ denotes the complete bipartite graph with two parts
of size $4$, the ambient group in the Platonov--Tavgen' construction
is itself a RAAG. Thus RAAGs need not be Grothendieck rigid or
absolutely profinitely rigid.

More recently, Corson--Hughes--M\"oller--Varghese proved that
$F_2\times F_2$ has infinite profinite genus and constructed
profinitely flexible RAAGs admitting no nontrivial free or direct
product decomposition; see~\cite[Corollaries~5.10 and~5.15]{CHMV26}.
\end{remark}

\bibliographystyle{amsalpha-author}
\bibliography{references}

@article{Droms87,
    author  = {Droms, Carl},
    title   = {Isomorphisms of Graph Groups},
    journal = {Proceedings of the American Mathematical Society},
    volume  = {100},
    number  = {3},
    year    = {1987},
    pages   = {407--408},
    doi     = {10.1090/S0002-9939-1987-0891135-8}
}

@article{KW16,
    author  = {Kropholler, Robert and Wilkes, Gareth},
    title   = {Profinite Properties of {RAAG}s and Special Groups},
    journal = {Bulletin of the London Mathematical Society},
    volume  = {48},
    number  = {6},
    year    = {2016},
    pages   = {1001--1007},
    doi     = {10.1112/blms/bdw056}
}

@article{CPZ25,
    author  = {Casals-Ruiz, Montserrat and Pintonello, Matteo and
               Zalesskii, Pavel},
    title   = {Pro-$\mathcal{C}$ {RAAG}s},
    journal = {Journal of Algebra},
    volume  = {664},
    year    = {2025},
    pages   = {177--208},
    doi     = {10.1016/j.jalgebra.2024.09.030}
}

@misc{ELMP26,
    author        = {Escart{\'i}n Ferrer, Marcos and Leoni, Giorgio and
                     Mart{\'i}nez P{\'e}rez, Conchita},
    title         = {Cohomology Rings and $p$-Local Behavior of Even
                     {Artin} Groups},
    year          = {2026},
    eprint         = {2606.30558},
    archivePrefix = {arXiv},
    primaryClass  = {math.GR},
    note          = {arXiv:2606.30558v1}
}

@book{Dav08,
    author    = {Davis, Michael W.},
    title     = {The Geometry and Topology of Coxeter Groups},
    series    = {London Mathematical Society Monographs Series},
    volume    = {32},
    publisher = {Princeton University Press},
    address   = {Princeton, NJ},
    year      = {2008},
    isbn      = {978-0-691-13138-2}
}

@article{BGP22,
    author  = {Blasco-Garc{\'i}a, Rub{\'e}n and Paris, Luis},
    title   = {On the Isomorphism Problem for Even {Artin} Groups},
    journal = {Journal of Algebra},
    volume  = {607},
    year    = {2022},
    pages   = {35--52},
    doi     = {10.1016/j.jalgebra.2020.05.025}
}

@article{Mal65,
    author  = {Mal'cev, A. I.},
    title   = {On the Faithful Representation of Infinite Groups by Matrices},
    journal = {American Mathematical Society Translations, Series 2},
    volume  = {45},
    year    = {1965},
    pages   = {1--18}
}

@article{CW02,
    author  = {Cohen, Arjeh M. and Wales, David B.},
    title   = {Linearity of {Artin} Groups of Finite Type},
    journal = {Israel Journal of Mathematics},
    volume  = {131},
    year    = {2002},
    pages   = {101--123},
    doi     = {10.1007/BF02785852}
}

@article{Digne03,
    author  = {Digne, Fran{\c{c}}ois},
    title   = {On the Linearity of {Artin} Braid Groups},
    journal = {Journal of Algebra},
    volume  = {268},
    number  = {1},
    year    = {2003},
    pages   = {39--57},
    doi     = {10.1016/S0021-8693(03)00327-2}
}

@article{BGMPP19,
    author  = {Blasco-Garc{\'i}a, Rub{\'e}n and
               Mart{\'i}nez-P{\'e}rez, Conchita and Paris, Luis},
    title   = {Poly-freeness of Even {Artin} Groups of {FC} Type},
    journal = {Groups, Geometry, and Dynamics},
    volume  = {13},
    number  = {1},
    year    = {2019},
    pages   = {309--325}
}

@article{Jan22,
  author  = {Jankiewicz, Kasia},
  title   = {Residual finiteness of certain 2-dimensional Artin groups},
  journal = {Advances in Mathematics},
  volume  = {405},
  year    = {2022},
  pages   = {108487},
  doi     = {10.1016/j.aim.2022.108487}
}

@article{DFPR82,
    author  = {Dixon, John D. and Formanek, Edward W. and
               Poland, John C. and Ribes, Luis},
    title   = {Profinite Completions and Isomorphic Finite Quotients},
    journal = {Journal of Pure and Applied Algebra},
    volume  = {23},
    number  = {3},
    year    = {1982},
    pages   = {227--231},
    doi     = {10.1016/0022-4049(82)90098-6}
}

@article{NS07a,
    author  = {Nikolov, Nikolay and Segal, Dan},
    title   = {On Finitely Generated Profinite Groups, I:
               Strong Completeness and Uniform Bounds},
    journal = {Annals of Mathematics},
    series  = {2},
    volume  = {165},
    number  = {1},
    year    = {2007},
    pages   = {171--238},
    doi     = {10.4007/annals.2007.165.171}
}

@article{NS07b,
    author  = {Nikolov, Nikolay and Segal, Dan},
    title   = {On Finitely Generated Profinite Groups, {II}:
               Products in Quasisimple Groups},
    journal = {Annals of Mathematics},
    series  = {2},
    volume  = {165},
    number  = {1},
    year    = {2007},
    pages   = {239--273},
    doi     = {10.4007/annals.2007.165.239},
    mrnumber = {2276770}
}

@phdthesis{VdL83,
    author  = {van der Lek, Harm},
    title   = {The Homotopy Type of Complex Hyperplane Complements},
    school  = {Katholieke Universiteit Nijmegen},
    address = {Nijmegen},
    year    = {1983}
}

@article{Paris97,
    author  = {Paris, Luis},
    title   = {Parabolic Subgroups of {Artin} Groups},
    journal = {Journal of Algebra},
    volume  = {196},
    number  = {2},
    year    = {1997},
    pages   = {369--399},
    doi     = {10.1006/jabr.1997.7098}
}

@book{RZ10,
    author    = {Ribes, Luis and Zalesskii, Pavel},
    title     = {Profinite Groups},
    edition   = {2},
    series    = {Ergebnisse der Mathematik und ihrer Grenzgebiete.
                 3. Folge},
    volume    = {40},
    publisher = {Springer},
    address   = {Berlin},
    year      = {2010},
    doi       = {10.1007/978-3-642-01642-4},
    isbn      = {978-3-642-01641-7}
}

@article{CHMV26,
    author  = {Corson, Samuel M. and Hughes, Sam and
               M{\"o}ller, Philip and Varghese, Olga},
    title   = {Higman--{Thompson} Groups and Profinite Properties of Right-Angled Coxeter Groups},
    journal = {Selecta Mathematica (New Series)},
    volume  = {32},
    year    = {2026},
    number  = {2},
    pages   = {28},
    doi     = {10.1007/s00029-026-01130-4}
}

@article{PT86,
  author  = {Platonov, V. P. and Tavgen', O. I.},
  title   = {Grothendieck's problem on profinite completions of groups},
  journal = {Soviet Mathematics Doklady},
  volume  = {33},
  pages   = {822--825},
  year    = {1986}
}

@misc{NB26,
      title={Braid groups are not profinitely rigid}, 
      author={Carl-Fredrik Nyberg-Brodda},
      year={2026},
      eprint={2607.20859},
      archivePrefix={arXiv},
      primaryClass={math.GR},
      url={https://arxiv.org/abs/2607.20859}, 
      note= {arXiv:2607.20859v1}
}

@article{Mol18,
  author  = {Moldavanskii, David},
  title   = {On the residual properties of Baumslag--Solitar groups},
  journal = {Communications in Algebra},
  volume  = {46},
  number  = {9},
  pages   = {3766--3778},
  year    = {2018},
  doi     = {10.1080/00927872.2018.1424867}
}

@article{Hum94,
  author  = {Humphries, Stephen P.},
  title   = {On representations of Artin groups and the Tits conjecture},
  journal = {Journal of Algebra},
  volume  = {169},
  number  = {3},
  pages   = {847--862},
  year    = {1994},
  doi     = {10.1006/jabr.1994.1312}
}

@article{HW99,
  author  = {Hsu, Tim and Wise, Daniel T.},
  title   = {On linear and residual properties of graph products},
  journal = {Michigan Mathematical Journal},
  volume  = {46},
  number  = {2},
  pages   = {251--259},
  year    = {1999},
  doi     = {10.1307/mmj/1030132408}
}

\end{document}